\documentclass[12 pt]{amsart}
\usepackage{amssymb, amsmath, amsthm}
\usepackage{enumitem}
\usepackage{mathrsfs}
\usepackage{hyperref}
\usepackage{bbm}
\usepackage{float}
\usepackage{caption}
\usepackage{tikz-cd}
\usepackage{bbm}
\usepackage{stmaryrd}

\DeclareMathOperator{\im}{im}
\DeclareMathOperator{\coker}{coker}
\DeclareMathOperator{\Pic}{Pic}
\DeclareMathOperator{\Spec}{Spec}

\DeclareMathOperator{\fieldchar}{char}
\DeclareMathOperator{\id}{id}

\DeclareMathOperator{\Br}{Br}
\DeclareMathOperator{\coh}{H}
\DeclareMathOperator{\Hom}{Hom}

\DeclareMathOperator{\Aut}{Aut}
\DeclareMathOperator{\lcm}{lcm}

\newtheorem{theorem}{Theorem}[section]
\newtheorem{lemma}[theorem]{Lemma}
\newtheorem{proposition}[theorem]{Proposition}
\newtheorem{corollary}[theorem]{Corollary}
\newtheorem*{theorem*}{Theorem}
\newtheorem*{corollary*}{Corollary}

\theoremstyle{definition}
\newtheorem{definition}[theorem]{Definition}
\newtheorem{note}[theorem]{Note}

\newtheorem{convention}[theorem]{Convention}
\newtheorem{example}[theorem]{Example}

\title{Tame Nodal Stacky Curves}
\author{Martin Bishop and William C. Newman}
\date{}

\begin{document}
\begin{abstract}
In this paper we analyze the properties of tame nodal stacky curves, which include
twisted curves and \textit{doubly-twisted} curves.
Our main results are a complete classification of the possible structures of a tame
stacky node, along with
computations of the Picard and Brauer groups of nodal stacky curves.
\end{abstract}

\maketitle

\section{Introduction}
The geometry of stacky curves is controlled in large part by their Picard and Brauer groups.
The Brauer group of a stacky curve is becoming well understood through the recent work
of Achenjang \cite{Ach24} and Bishop \cite{Bis25}, and the Picard group of a \textit{smooth} stacky
curve was computed by Lopez \cite{Lop23}.
This paper aims to understand
how the presence of stacky structure at nodes affects the Picard group.
We will also finish a Brauer group computation which was promised in \cite[Example 3.9]{Bis25}.

This paper gives an explicit description of the Picard and Brauer groups for tame nodal stacky curves,
covering the two possible local structures at a node: \textit{twisted} nodes, where the geometric stabilizer
is cyclic and preserves the branches; and \textit{doubly-twisted}, where there is a cyclic subgroup acting
on each branch as well as a $\mathbb Z/2$ swapping branches.

\subsection{Main results}
We begin by stating a version of Theorem 1.1 of \cite{Lop23}.
\begin{theorem*}[\cite{Lop23}, Theorem 1.1]
    Let $\mathcal C$ be a stacky curve with a smooth point $p\in\mathcal C$ with stabilizer group $\mu_n$. Let $C$ be the coarsening
of $\mathcal C$ at $p$. Then we have
\[\Pic(\mathcal C)=\Pic(C)\langle \sqrt[n]{\mathcal O(p)}\rangle.\]
\end{theorem*}
Here, by $\Pic(C)\langle \sqrt[n]{\mathcal O(p)}\rangle$, we mean
$$
\frac{\Pic(C)\oplus \mathbb Z}{\langle(\mathcal O(p),0)-n(0,1)\rangle},
$$
and we identify $p\in\mathcal C$ with the corresponding point in $C$.

At a stacky node, whether twisted or doubly-twisted, there is an action of $\mu_n$, which is the largest subgroup of the stabilizer group that fixes the branches. The action can be written locally as $\zeta_n\cdot (x,y)\mapsto (\zeta_n x,\zeta_n^a y)$ for $\zeta_n$ a primitive root of unity. Our computation of the Picard group depends on these numbers $a$ and $n$. 

\begin{theorem*}[Theorem~\ref{twisted Picard v1}]
Let $\mathcal C$ be a stacky curve.
Assume that $\mathcal C$ has a twisted node $p$ and denote its partial coarsening
at $p$ by $C$. Let $\tilde{\mathcal C}$
be the partial normalization of $\mathcal C$ at $p$.
Let $\tilde C$ be the coarsening
of $\tilde{\mathcal C}$ at the preimages of $p$.
If $p$ is nonseparating, we have an exact sequence
$$
0\rightarrow\mathbbm k^{\times}\rightarrow\Pic\mathcal C\rightarrow \Pic\tilde C\left<\sqrt[n]{\mathcal O(p_1+ap_2)}\right>\rightarrow 0.
$$
If $p$ is separating, we have
\[\Pic(\mathcal C)=\Pic\tilde C\left<\sqrt[n]{\mathcal O(p_1+ap_2)}\right>.\]
\end{theorem*}

The sequence for nonseparating nodes does not always split, though it will if $\mathbbm k=\bar{\mathbbm k}.$ An important special case is when the node is balanced.

\begin{corollary*}[Corollary \ref{balanced twisted Picard}]
In the above setting, assume also that $\mathbbm k$ is algebraically closed, and that the node $p$ is \textit{balanced}, that is, $a\equiv-1\mod n$. If $p$ is nonseparating, then
$$
\Pic\mathcal C\cong\mathbbm k^{\times}\oplus\Pic\tilde{C}\oplus\mathbb Z/n\cong\Pic C\oplus\mathbb Z/n,
$$
and if $p$ is separating
\[\Pic\mathcal C\cong \Pic\tilde C\oplus \mathbb Z/n= \Pic C\oplus \mathbb Z/n.\]
\end{corollary*}

In Section \ref{doubly-twisted section}, we address what happens at \textit{doubly-twisted} nodes. For the exact statement of the following theorem, see Theorem \ref{doubly-twisted Picard}.
\begin{theorem*}[Theorem \ref{doubly-twisted Picard}]
Let $\mathcal C$ be a stacky curve.
Assume that $\mathcal C$ has a doubly-twisted node $p$.
Let $C$ be its coarsening at $p$.
Then either
\[\Pic(\mathcal C)= \mu_2\oplus \Pic(C)\left\langle\sqrt[d_-]{\mathcal O(p)}\right\rangle\]
or
\[\Pic(\mathcal C)=  \Pic(C)\langle\sqrt[2d_-]{\mathcal O(2p)}\rangle\]
depending on the structure of the stabilizer.
\end{theorem*}

Lastly, we also compute the contribution of a node to the Brauer group when
the base field is algebraically closed. See Proposition \ref{H2 doubly twisted}
for the exact statement.
\begin{theorem*}[Propositions \ref{H2 doubly twisted} and \ref{H2 twisted}]
Let $\mathcal C$ be a stacky curve over an algebraically closed field $\mathbbm k$.
The contribution of a twisted node to the Brauer group of $\mathcal C,$
$\coh^2(G,\mathbbm k^{\times}_{\text{triv}}),$ is $0$.
The contribution of a doubly-twisted node, $\coh^2(G,\mathbbm k^{\times}_{\text{triv}}),$ is either $0$ or $\mathbb Z/2$, depending on the structure of the stabilizer.
\end{theorem*}

These formulas track precisely how the introduction of stackiness alters the behavior of the classical
exact sequence
$$
0\rightarrow\mathbbm k^{\times}\rightarrow\Pic C\rightarrow\Pic\tilde C\rightarrow0
$$
where $C$ is a curve with a nonseparating node $p$ and $\tilde C$ is the partial normalization at $p$.
They also contrast with the case of a smooth stacky curve. In \cite{Lop23}, Lopez shows that
a smooth stacky curve is always obtained as a root stack from its coarse moduli space. This
provides a uniform formula for computing Picard groups: it is always the pushout described in
\cite[Theorem 1.1]{Lop23}. Whereas our Theorems \ref{doubly-twisted Picard} and \ref{twisted Picard v1}
show that the Picard groups of nodal stacky curves,
while captured in concise formulas, vary depending on the stack
structure at the node and cannot be handled as uniformly. In particular, our results agree with
the established fact that stacky nodes are not root stacks.

\subsection{Conventions}\label{conventions}
We work over a field $\mathbbm k$, not assumed to be algebraically closed except in
the context of Brauer groups.

We take a \textit{stacky curve} to be a tame, proper, geometrically connected, one-dimensional Deligne-Mumford
stack with trivial generic stabilizer of finite type over $\mathbbm k$.
We assume that all of our nodes are \textit{split}, meaning that
each node as well as their preimages in the normalization are $\mathbbm k$-points.
Note that the tameness hypothesis means that $\fieldchar\mathbbm k$ is always coprime
to $n$ and is coprime to 2 when discussing doubly-twisted nodes.

\subsection{Acknowledgments}
The first author would like to thank Giovanni Inchiostro, Cherry Ng, Brian Nugent,
and Alex Schefflin for their helpful conversations.

The authors would like to thank Jarod Alper and Yuchen Liu,
who helped diagnose a persistent error with an early version of
Theorem \ref{twisted Picard v1}, as well as Dan Abramovich, who pointed
them to \cite{Eke95}.

The authors met at the 2023 AGNES Summer School on Intersection Theory on Moduli Spaces. The authors thank the organizers and speakers.

The second author was supported by the National Science Foundation under Grant No. DMS-2231565.

\section{Background}
\subsection{Assorted background}
We recall Tsen's Theorem for curves, as well as a stacky version.
\begin{theorem}[Tsen's Theorem]\label{Tsen}
Suppose that $C$ is a curve over an algebraically closed field. Then
$\coh^2(C,\mathbb G_m)=0$.
\end{theorem}

\begin{theorem}\label{Brauer group}
Let $\mathcal C$ be a stacky curve with trivial generic stabilizer over an algebraically closed field $K$.
Let $p_1,\dots,p_n$ be the singular points of $\mathcal C_{\text{red}}$ and $G_1,\dots,G_n$
be the geometric stabilizers at these points. Then
$$
\coh^2(\mathcal C,\mathbb G_m)=\bigoplus_{i=1}^n\coh^2(G_i,K^{\times})
$$
where the action of each $G_i$ on $K^{\times}$ is trivial.
\end{theorem}

\begin{proof}
See \cite[Proposition 3.7]{Bis25}.
\end{proof}

Next, we describe the group cohomology of cyclic groups.
\begin{proposition}\label{group coho}
Let $G$ be a cyclic group of order $n$ and $M$ a $G$-module.
Let $x$ denote the generator of $G$. Let
$$
N=1+x+x^2+\dots+x^{n-1}\in\mathbb ZG.
$$
Then
$$
\coh^k(G,M)=
\begin{cases}
M^G/NM & \text{if $k\geq2$ even}\\
\ker(\cdot N)/(x-1)M & \text{if $k$ odd}.
\end{cases}
$$
In particular, if the action of $G$ on $M$ is trivial then
$\coh^k(G,M)=\Hom(G,M)$ if $k$ is odd and $\coh^k(G,M)=0$ if $k$ is even and multiplication by
$|G|$ is surjective on $M$.
\end{proposition}

\begin{proof}
See any standard reference for group cohomology, such as \cite[Chapter 17]{DF04} or
\cite[Chapter 6]{Wei94}.
\end{proof}

\subsection{Extending a result of Meier}

The main result of this section is Corollary \ref{Meier stalk}, which was
originally stated in the $n=2$ case in \cite{Mei18} and then generalized
to the $n\geq2$ case in \cite{Ach24} and \cite{Bis25}. In \cite{Ach24}, conditions are given for when it
applies for $n\geq1$. We now upgrade this to all cases for $n\geq1$, using Lemma \ref{key lemma}.

\begin{definition}[\cite{ATW20}, 2.3.1]\label{coarsening}
    A morphism $f:X\rightarrow Y$ of stacks is called a \textit{coarsening}
    if the inertia $\mathcal I_{X/Y}$ is finite over $X$ and for any flat $Z\rightarrow Y$
    from an algebraic space $Z$ the resulting base change
    $X\times_YZ\rightarrow Z$ is a coarse moduli space.
\end{definition}

\begin{proposition}[\cite{ACV03} Proposition A.0.1, \cite{Mei18} Proposition 8]\label{Meier structure}
    Let $\mathcal X$ be a separated tame Deligne-Mumford stack of finite type
    over $\mathbbm k$ with coarsening $f:\mathcal X\rightarrow X$.
    Let $x:\Spec K\rightarrow X$ be a geometric point. Let $G=(\mathcal I_{\mathcal X/X})_x$ be
    the relative stabilizer.
    Then for all $n\geq 0$
    $$
    (\mathbf R^nf_*\mathbb G_m)_x=\coh^n(G,R^{\times})
    $$
    for some strictly Henselian local ring $R$ with residue field $K$. Moreover, the group
    $G$ acts trivially on $K$.
\end{proposition}

\begin{proof}
See \cite[Proposition 2.11]{Bis25}.
\end{proof}

\begin{lemma}\label{key lemma}
Let $A$ be a $G$-module. Let $p$ be a prime that is coprime to $|G|$.
Then for all $n\geq1$,
$$
\coh^n(G,A)=\coh^n(G,A[1/p]).
$$
\end{lemma}

\begin{proof}
Let $n\geq 1$. 
Let $f:A\rightarrow A[1/p]$ be the natural map.
Consider the sequence
$$
0\rightarrow\ker f\rightarrow A\rightarrow A[1/p]\rightarrow\coker f\rightarrow0.
$$
This leads to two short exact sequences
\begin{equation}\label{A ker sequence}
0\rightarrow\ker f\rightarrow A\rightarrow \im f\rightarrow 0
\end{equation}
and
\begin{equation}\label{A coker sequence}
0\rightarrow\im f\rightarrow A[1/p]\rightarrow\coker f\rightarrow 0.
\end{equation}

Taking cohomology for Equation \eqref{A ker sequence} yields
$$
\coh^n(G,\ker f)\rightarrow\coh^n(G,A)\rightarrow\coh^n(G,\im f)\rightarrow\coh^{n+1}(G,\ker f).
$$
Note that $f$ is injective on elements which are coprime to $p$, and hence
$\ker f$ is a $p$-group. Therefore $\coh^i(G,\ker f)=0$ for all $i\geq 1$, since $|G|$ is coprime to $p$.
Thus
$$
\coh^n(G,A)=\coh^n(G,\im f).
$$

Repeating this for Equation \eqref{A coker sequence} gives
$$
\coh^{n-1}(G,\coker f)\rightarrow\coh^n(G,\im f)\rightarrow\coh^n(G,A[1/p])\rightarrow
\coh^n(G,\coker f).
$$
Now, the cokernel of $f$ is also a $p$-group, and hence
$$
\coh^n(G,\coker f)=0.
$$
Since $\coh^n(G,\im f)$ contains no $p$-torsion, we see
that the image of the $p$-group $\coh^{n-1}(G,\coker f)$ is trivial
(note that we are \textit{not} saying
that $\coh^{n-1}$ vanishes, since it may not in the case $n=1$). Therefore
\[\coh^n(G,A)=\coh^n(G,\im f)=\coh^n(G,A[1/p]).\qedhere\]
\end{proof}

\begin{corollary}[\cite{Mei18} Lemma 9, \cite{Ach24} Lemma 4.1, \cite{Bis25} Corollary 2.12]\label{Meier stalk}
Let $\mathcal X$ and $G$ be as in Proposition \ref{Meier structure}. Then for all geometric points
$x:\Spec K\rightarrow X$ and for all $n\geq 1$,
$$
(\mathbf R^nf_*\mathbb G_m)_x=\coh^n(G,K^{\times}).
$$
\end{corollary}

\begin{proof}[Proof of Corollary \ref{Meier stalk}]
This proof is mostly an exact reproduction of the proof of Corollary 2.12 in \cite{Bis25}, but now extended
to work for all $n\geq1$.

If $p=\fieldchar K>0$, let $p=\fieldchar K$, and otherwise set $p=1$.
Using Proposition \ref{Meier structure}, all that's left to show is that $\coh^n(G,R^{\times})=
\coh^n(G, K^{\times})$.
Consider the exact sequence
$$
1\rightarrow\ker\rightarrow R^{\times}[1/p]
\rightarrow K^{\times}[1/p]\rightarrow 1.
$$
Since $R$ and $K$
are both strictly Henselian, we have that  $R^{\times}[1/p]$ and
$K^{\times}[1/p]$ are both divisible. Therefore $\ker$ is divisible.
Now by \cite[Tag 06RR]{Sta},
the torsion of $R^{\times}[1/p]$ maps isomorphically onto the torsion of
$K^{\times}[1/p]$, and hence $\ker$ is torsion-free. The structure theorem
for divisible groups (see \cite[Theorem 23.1]{Fuc70}) then implies that $\ker$ is a $\mathbb Q$-vector space
and hence has vanishing higher cohomology.

Therefore, taking cohomology, we see that
$$
\coh^n(G,R^{\times})=
\coh^n(G,R^{\times}[1/p])=
\coh^n(G,K^{\times}[1/p])=
\coh^n(G,K^{\times})
$$
for all $n\geq 1$, where the first and last equalities use Lemma \ref{key lemma} if $p>1$.
\end{proof}

\section{Nodal stacky curves}
We define nodal stacky curves. Recall from Section \ref{conventions}
that our stacky curves are tame, proper, one-dimensional, geometrically connected Deligne-Mumford stacks with trivial generic stabilizer of finite type over $\mathbbm k$.

\begin{definition}\label{stacky nodal curve def}
Given a stacky curve $\mathcal C$ over $\mathbbm k$, a point $p\in \mathcal C$ is a (split) node if there exists an \'etale neighborhood $U\to \mathcal C$ with $U$ a scheme such that there exists $q\in U(k)$ mapping to $p$ with $q$ a split node of $U$. Equivalently, the completion of the strictly Henselian local ring at $p$ is $\bar{\mathbbm k}[[x,y]]/(xy)$ and the preimages of $p$ in the normalization are $\mathbbm k$-rational.

A \textit{nodal stacky curve} $\mathcal C$ is a stacky curve so that every point $p\in \mathcal C$ is either a node or is smooth.
\end{definition}

The next result follows from and extends \cite{Eke95}, where its content is distributed throughout Section 1.
\begin{proposition}\label{stacky nodal structure}
Let $\mathcal C$ be as above, and $p\in\mathcal C$ a node.
Let $R$ denote the strict Henselization of the local ring of $\mathbbm k[x,y]/(xy)$ at the origin.
Then
$\mathcal C^{sh}=\mathcal C\times_C\Spec\mathcal O_{C,p}^{sh}$ is of the form 
$[\Spec R/G]$
for $G$ equal to one of either:
\begin{itemize}
\item $\mu_n$, where $\mu_n$ acts on $(x,y)$ as $(\zeta_n x,\zeta_n^a y)$ for a primitive
$n^{\text{th}}$ root of unity $\zeta_n$ and an $a\in(\mathbb Z/n)^{\times}$; or
\item an extension of $\mathbb Z/2$ by $\mu_n$, where $\mathbb Z/2$ interchanges the axes and
$\mu_n$ acts on $(x,y)$ as above with the additional condition
that $a^2\equiv1\mod n$.
\end{itemize}
\end{proposition}

\begin{proof}
The statement that $\mathcal C^{sh}$ is of the form $[\Spec R/G]$ for some finite $G$ follows
from the assumption that $\mathcal C$ has a node at $p$.

Considering the action of $G$ on the branches of $\mathcal C$ gives a map
$G\rightarrow\mathbb Z/2$, whose kernel consists of the branch-preserving elements of $G$.
This kernel is isomorphic to some $\mu_n$, since the only
geometric stabilizers of a smooth tame curve are cyclic.
Therefore $G$ is either $\mu_n$ (in the case where all elements are
branch-preserving) or an extension of $\mathbb Z/2$ by $\mu_n$
where $\mathbb Z/2$ acts by exchanging branches
(in the case where there are branch-swapping elements).

Because $\mathcal C$ has trivial generic stabilizers,
$\zeta_n$ has to act as $(\zeta_nx,\zeta_n^ay)$ (without loss of generality) for some primitive
$n^{\text{th}}$ root of unity $\zeta_n$ and $a$ coprime to $n$,
else the $y$-branch will be stabilized everywhere by $\zeta_n^{n/\gcd(a,n)}$.
In the case where $G$ is an extension of $\mathbb Z/2$ by $\mu_n$,
we have the following further limitation on $a$.
Let $\sigma\in G$ denote a lift of the generator of $\mathbb Z/2$. Then
$$
\sigma\zeta_n\sigma^{-1}\cdot(x,y)=(\zeta_n^ax,\zeta_ny)
$$
is of the form $(\zeta_n^kx,\zeta_n^{ak}y)$ for some $k$. Hence $k=a$ and $ak=a^2\equiv1
\mod n$, as desired.
\end{proof}

For the rest of the paper, the following numbers defined in terms of $a$ and $n$ will be useful
$$
d_-=\gcd(a-1,n),\quad d_+=\gcd(a+1,n),\quad\text{and}\quad n_+=n/d_+.
$$

When $p\in \mathcal C$ is a node so that the stabilizer group $G$ acts trivially on the branches, $p$ is called a \textit{twisted node}. If all nodes in a nodal stacky curve $\mathcal C$ are twisted, $\mathcal C$ is called a \textit{twisted curve}.
Such curves have been studied extensively in the literature
(see \cite{AV02, AOV11, ACV03} among many others).
While we allow $a$ to take on any value relatively prime to $n$, most
recent literature has focused on the \textit{balanced} case where
$a\equiv-1\mod n$ (in fact, many authors now assume the balanced condition in the definition of twisted curves). 

A twisted node with any value of $a$ coprime to $n$ is possible. The action of
$\mu_n$ on $\Spec(\mathbbm k[x,y]/(xy))$ by
$\mu_n \cdot \zeta_n (x,y)=(\zeta_nx,\zeta_n^ay)$ for $a$ coprime to $n$
extends to its the closure in $\mathbb P^1\times \mathbb P^1$. The quotient of this
closure by $\mu_n$ has the desired structure at the node.

Now, we consider the other types of stacky nodes.
\begin{definition}
For a node $p\in \mathcal C$ whose stabilizer group $G$ switches the branches, we say that $p\in \mathcal C$ is
a \textit{doubly-twisted node}.
\end{definition}

The particular case of a doubly-twisted node with $G=\mu_n\rtimes\mathbb Z/2$ and $a=-1$ (hence
$G\cong D_n$) was discussed in \cite[Section 6]{Rom05}, where such a node was said to
be of \textit{dihedral type}.

\subsection{Numerology of doubly-twisted curves}\label{numerology section}
For a doubly-twisted curve, the geometric stabilizer group $G$ is by definition an
extension of $\mathbb Z/2$ by $\mu_n$. We will work out some additional
numerical properties of $G$ beyond those of Proposition \ref{stacky nodal structure}.

We have already established that the action determines $a\in\mathbb Z/n$
such that $a^2\equiv 1\mod n$. We can extract an additional piece of data from this.
Let $\sigma\in G$ be
a lift of the generator of $\mathbb Z/2$. Then $\sigma^2\in\mu_n\leq G$, and so, fixing a primitive
root $\zeta_n\in\mu_n$, we have $\sigma^2=\zeta_n^m$ for some $m$. It turns out that $m$ is
exactly the class which determines the extension class $[G]$.

\begin{proposition}
For the action of $\mathbb Z/2$ on $\mu_n$ determined by conjugation, we have
$$
\coh^2(\mathbb Z/2,\mu_n)=\frac{\ker(1-a)}{(1+a)\mu_n}.
$$
\end{proposition}

\begin{proof}
This follows from Proposition \ref{group coho}.
\end{proof}

\begin{proposition}\label{m prop}
The integer $m$ is well-defined up to addition by $1+a$, is in the kernel of multiplication by $1-a$, and determines the extension
class of $[G]$. Moreover, the extension is trivial if and only if $d_+|m$.
\end{proposition}

\begin{proof}
Since $\sigma^2$ is fixed by conjugation, and conjugation acts as multiplication by $a$,
we see that $m$ is in the kernel of multiplication by
$1-a$ in $\mu_n$, and is well-defined up to $(1+a)\mu_n$. Hence
it determines a class in $\coh^2(\mathbb Z/2,\mu_n)$, and this class corresponds precisely to
$[G]$. To see the last statement, simply observe that the trivial extensions correspond to
\[
m\in(1+a)\mathbb Z/n=d_+\mathbb Z/n. \qedhere
\]
\end{proof}

\begin{proposition}\label{a m prop}
The group $G$ is determined completely by $a$ and $m$.
\end{proposition}

\begin{proof}
The element $a$ determines the action of $\mathbb Z/2$ on $\mu_n$.
For a fixed action of $\mathbb Z/2$ on $\mu_n$, the extensions of $\mathbb Z/2$
by $\mu_n$ with the given action are parametrized by $\coh^2(\mathbb Z/2,\mu_n)$.
By Proposition \ref{m prop} $m$ determines the class of $[G]$ in
$\coh^2(\mathbb Z/2,\mu_n)$.
\end{proof}

Given $n,a,m$ under these constraints, we can construct a stacky curve whose stabilizer group is the group $G$ coming from this data. The action of $(\mathbbm k^\times)^2\ltimes \mathbb Z/2$ on $\Spec(\mathbbm k[x,y]/(xy))$ given by
\[(a,b)\cdot (x,y)=(ax,by)\]
for $a,b\in\mathbbm k^\times$ and 
\[\tau\cdot (x,y)=(y,x),\]
where $\tau$ is the generator of $\mathbb Z/2,$ extends to its closure in $\mathbb P^1\times \mathbb P^1$. The quotient of this closure by the subgroup generated by $(\zeta_n,\zeta_n^a),\tau(1,\zeta_n^m)$ has the desired stack structure at the node.

Lastly, we will need the following result later. 
\begin{proposition}\label{d_-d_+}
We have that $d_-d_+=n$ or $2n$.
\end{proposition}

\begin{proof}
Let $g=\gcd(d_-,d_+)$. Then $g$ divides $(a+1)-(a-1)=2$, and so $g=1$ or 2.
Since
$$
\lcm(d_-,d_+)=\frac{d_-d_+}g
$$
divides $n$ (as both $d_-$ and $d_+$ do), we see that $d_-d_+$ divides $gn$, which is either $n$ or $2n$.

Because $n$ divides $a^2-1$, we have $n|d_-d_+$. Hence $d_-d_+$ is either $n$ or $2n$.
\end{proof}

\section{Doubly-twisted nodes}\label{doubly-twisted section}

\begin{convention}\label{doubly twisted conventions}
Throughout this section we fix a tame stacky curve $\mathcal C$ over a field
$\mathbbm k$
which has a doubly-twisted node $p: BG\to  \mathcal C$,
as well as its coarsening at $p$, $\pi:\mathcal C\rightarrow C$.
Let
$\eta:\tilde{\mathcal C}\rightarrow\mathcal C$ be the partial normalization at $p$.
Lastly, recall that we have fixed
$d_+=\gcd(a+1,n)$, $d_-=\gcd(a-1,n)$, and $n_+=n/d_+$.
\end{convention}

\subsection{The Leray spectral sequence}
Following \cite{Bis25}, we use the Leray spectral sequence for $\pi:\mathcal C\rightarrow C$
$$
\coh^p(C,\mathbf R^q\pi_*\mathbb G_m)\implies\coh^{p+q}(\mathcal C,\mathbb G_m).
$$
Since $\pi_*\mathbb G_m=\mathbb G_m$ for coarsenings,
the sequence of low-degree terms for this yields
\[
0\rightarrow\Pic C\rightarrow\Pic\mathcal C\rightarrow
\coh^0(C,\mathbf R^1\pi_*\mathbb G_m)\rightarrow \coh^2(C,\mathbb G_m)\rightarrow\coh^2(\mathcal C,\mathbb G_m),
\]
using that $\coh^1(X,\mathbb G_m)=\Pic X$. 

By Corollary \ref{Meier stalk}, $\mathbf R^1\pi_*\mathbb G_m$ vanishes away from the node,
where its stalk is
$$
(\mathbf R^1\pi_*\mathbb G_m)_p=\coh^1(G,\bar{\mathbbm k}^{\times}_{\text{triv}})
$$
with the trivial action of $G$ on $\mathbbm k^{\times}$. Since
$$
\coh^1(G,\bar{\mathbbm k}^{\times}_{\text{triv}})=\Hom(G,\bar{\mathbbm k}^{\times}_{\text{triv}})=
\Hom(G^{\text{ab}},\bar{\mathbbm k}^{\times})
$$
the below sequence is exact
\[
0\rightarrow\Pic C\rightarrow\Pic\mathcal C\rightarrow
\Hom(G^{\text{ab}},\bar{\mathbbm k}^{\times})\rightarrow\coh^2(C,\mathbb G_m)\rightarrow\coh^2(\mathcal C,\mathbb G_m).
\]

\begin{lemma}\label{normalization}
The preimage of $p$ in $\tilde{\mathcal C}$ is a single point with
stabilizer $\mu_n$, and $C$ is smooth at the image of $p$.
Moreover, $C$ is the coarsening of $\mathcal C$ at $p$.
\end{lemma}

\begin{proof}
The \'{e}tale local structure of $\mathcal C$ at the node is
$$
\left[\frac{\Spec\mathbbm k[x,y]/(xy)} G\right]
$$
where $G$ swaps the branches,
and so the \'{e}tale local structure of the normalization of $\mathcal C$ is
$$
\left[\frac{\Spec\mathbbm k[z]}{\mu_n}\right].
$$
That is, it is smooth with geometric stabilizer $\mu_n$.
Moreover, the composition $\tilde{\mathcal C}\to \mathcal C\to C$
is also the coarsening of $\tilde{\mathcal C}$.
\end{proof}

\begin{lemma}\label{H2 injectivity}
The following sequence is exact
\begin{equation}\label{Leray}
0\rightarrow\Pic C\rightarrow\Pic\mathcal C\rightarrow
\Hom(G^{\text{ab}},\bar{\mathbbm k}^{\times})\rightarrow0.
\end{equation}
\end{lemma}

\begin{proof}
It suffices to show that $\coh^2(C,\mathbb G_m)\to \coh^2(\mathcal C,\mathbb G_m)$ is injective.
Away from the node, we have $\mathcal C\setminus\{p\}=U=C\setminus\{\pi(p)\}$.
Consider the diagram
$$
\begin{tikzcd}
\coh^2(C,\mathbb G_m)\arrow[rd]\arrow[rr] & & \coh^2(\mathcal C,\mathbb G_m)\arrow[ld]\\
 & \coh^2(U,\mathbb G_m). &
\end{tikzcd}
$$
The restriction map to $U$ is injective since the image of $p\in C$ is a smooth point (by Lemma
\ref{normalization}). Commutativity of the diagram then
forces the pullback to $\mathcal C$ to be injective as well.
\end{proof}

Next, we determine $G^{\text{ab}}$.

\begin{proposition}\label{Gab}
   The abelianization of $G$ is given by
    \[G^{\text{ab}}\cong
\begin{cases}
	\mathbb Z/2\oplus\mathbb Z/d_-, & \text{if $m$ is even}\\
	\mathbb Z/2d_-, & \text{if $m$ is odd.}
\end{cases}.\]
\end{proposition}
Note that the two groups are isomorphic when $d_-$ is odd.

\begin{proof}
We know that $G$ fits into the following exact sequence
$$
1\rightarrow\mu_n\rightarrow G\rightarrow\mathbb Z/2\rightarrow1.
$$
Denote a generator of $\mu_n$ by $t$ and a lift of the generator of $\mathbb Z/2$ to $G$ by $\sigma$.
Then $G$ has the following relations:
$$
t^n=1\quad \sigma t\sigma^{-1}=t^a\quad \sigma^2=t^m,
$$
subject to the numerical conditions from Section \ref{numerology section}.
Abelianizing $G$ adds the relation $\sigma t=t\sigma$.

If we consider the free abelian group
$\mathbb Z\left<T\right>\oplus\mathbb Z\left<\Sigma\right>\cong\mathbb Z^2$, then the above says
that $G^{\text{ab}}$ is the quotient of this group by the relations
\begin{itemize}
\item $nT=0$, from $t^n=1$ in $G$,
\item $(a-1)T=0$, from $t=\sigma t\sigma^{-1}=t^a$ in $G^{\text{ab}}$, and
\item $2\Sigma-mT=0$, from $\sigma^2=t^m$ in $G$.
\end{itemize}
That is, $G^{\text{ab}}$ is the quotient of $\mathbb Z^2$ by the $\mathbb Z$-span of the rows of
$$
\begin{pmatrix}
n & 0\\
a-1 & 0\\
-m & 2
\end{pmatrix}.
$$
We may row-reduce the first two rows so that the top left entry is $\gcd(a-1,n)$, i.e. $d_-$.
With this, $G^{\text{ab}}$ is now the quotient of $\mathbb Z^2$ by the $\mathbb Z$-span of the rows of
$$
\begin{pmatrix}
d_- & 0\\
-m & 2
\end{pmatrix}.
$$
Set $k=\gcd(d_-,m,2)$.
Putting this matrix into Smith Normal Form shows that the span is precisely
\[
\mathbb Z/k\oplus\mathbb Z/(2d_-/k)=
\begin{cases}
\mathbb Z/2\oplus\mathbb Z/d_-, & \text{if $2$ divides both $d_-$ and $m$}\\
\mathbb Z/(2d_-), & \text{else}.
\end{cases}
\]
This is equivalent to the statement given in the proposition after
noticing that both groups are isomorphic when $d_-$ is odd.
\end{proof}

Because $\fieldchar\mathbbm k$ is coprime to $2n$,
we have $\Hom(G^{\text{ab}},\bar{\mathbbm k}^\times)\cong G^{\text{ab}}$,
and then $\Pic\mathcal C$ is an extension of a group of size $2d_-$ by $\Pic C$
by Lemma~\ref{H2 injectivity}. 

\begin{definition}
Define the sheaf $Q$ on $\mathcal C$ as the following quotient:
$$
0\rightarrow\mathcal O_{\mathcal C}^{\times}\rightarrow\eta_*\mathcal O_{\tilde{\mathcal C}}^{\times}\rightarrow
Q\rightarrow 0.
$$
\end{definition}
Since pushforward along a finite morphism is exact in the \'{e}tale topology,
the above short exact sequence gives the long exact sequence
$$
0\rightarrow\mathcal O_{\mathcal C}(\mathcal C)^{\times}\rightarrow\mathcal O_{\tilde{\mathcal C}}(\tilde{\mathcal C})^{\times}
\rightarrow \coh^0(\mathcal C,Q)\rightarrow
\Pic\mathcal C\rightarrow\Pic\tilde{\mathcal C}.
$$

The \'etale local description of a doubly-twisted node shows that $\tilde{\mathcal C}$ is still geometrically connected, that is, a doubly-twisted node is
nonseparating. Then, because $\mathcal C$ and $\tilde{\mathcal C}$ are both geometrically connected and proper, we have
$\mathcal O_{\mathcal C}(\mathcal C)^{\times}=\mathcal O_{\tilde{\mathcal C}}(\tilde{\mathcal C})^{\times}=\mathbbm k^{\times}$.
Therefore we have the following exact sequence:
\begin{equation}\label{exact sequence}
0
\rightarrow \coh^0(\mathcal C,Q)\rightarrow
\Pic \mathcal C\xrightarrow{\eta^*}\Pic\tilde{\mathcal C}.
\end{equation}

We now begin an analysis of the terms of Equation \eqref{exact sequence}.

\begin{proposition}\label{Q structure}
Let $\mathbbm k^{\times}_{\text{inv}}$ denote the sheaf on $BG$ associated
to the $G$-module whose underlying group is $\mathbbm k^\times$ and whose action inverts the elements of $\mathbbm k^\times$, corresponding to $G\to \mathbb Z/2\to \Aut({\mathbbm k}^\times)$. Then $Q\cong p_* \mathbbm k^\times_{\text{inv}}.$
\end{proposition}

\begin{proof}
By definition $Q(U)$ is the sheafification of the assignment
$$
U\mapsto\frac{(\eta_*\mathbb G_m)(U)}{\mathbb G_m(U)}=\frac{\mathbb G_m(\tilde U)}{\mathbb G_m(U)}.
$$
This is now reduced to the level of curves, where we know that this quotient is, Zariski locally, just
$(\mathbbm k^{\times})^m$ where $m$ is the number of preimages of the node. But the assignment
$U\mapsto(\mathbbm k^{\times})^m$ is a familiar sheaf: namely, it is the pushforward of
$\mathbbm k^{\times}$ from the residual gerbe at the node, as desired.

Now we must determine the action of $G$. We note that the $\mathbbm k^{\times}$
appearing in $(\mathbbm k^{\times})^m$ is actually
$$
\frac{\mathbbm k^{\times}\times\mathbbm k^{\times}}{\mathbbm k^{\times}}
$$
where the quotient is by the diagonal $(t,t)$. Since the branch-preserving part of $G$ acts
trivially on $\mathbbm k^{\times}$, it acts
trivially on the quotient.
However, the branch-swapping part of $G$
does \textit{not} act trivially. More specifically,
the $\mathbb Z/2$ interchanges the two copies of $\mathbbm k^{\times}$
and hence acts by inversion on the quotient, as claimed.
\end{proof}

\begin{corollary}
The global sections of $Q$ are $\coh^0(\mathcal C,Q)=\mu_2$.
\end{corollary}

\begin{proof}
Exactness of pushforward under a closed immersion together with Proposition \ref{Q structure} implies that
$$
\coh^0(\mathcal C,Q)=\coh^0(BG,\mathbbm k^{\times}_{\text{inv}})=
\coh^0(G,\mathbbm k^{\times}_{\text{inv}})=\mu_2.
$$
Where the last equality holds since inversion fixes $\mu_2$. 
\end{proof}

Hence, our exact sequence is 
\begin{equation}\label{final exact sequence}
0\rightarrow\mu_2\rightarrow\Pic\mathcal C\xrightarrow{\eta^*}
\Pic\tilde{\mathcal C}.
\end{equation}
\begin{theorem}\label{doubly-twisted Picard}
Let $\mathcal C$ be a stacky curve with coarse moduli space $C$.
Assume that $\mathcal C$ has a doubly-twisted node $p$.
Let $C$ be its coarsening at $p$.
Then 
\[\Pic(\mathcal C)=\begin{cases}
    \mu_2\oplus \Pic(C)\left<\sqrt[d_-]{\mathcal O(p)}\right>, &
    \text{if $m$ is even} \\
    \Pic(C)\left<\sqrt[2d_-]{\mathcal O(2p)}\right>, & \text{if $m$ is odd.}
\end{cases}\]
\end{theorem}
We note that, as in Proposition~\ref{Gab}, there is overlap in
these cases when $d_-$ is odd: multiplying a generator of $\mu_2$ with
$\sqrt[d_-]{\mathcal O(p)}$ gives a $2d_-$-th root of $\mathcal O(2p)$ if $d_-$ is
odd, and so these descriptions are equivalent in this case.
\begin{proof}
    We have a commutative diagram
    \begin{center}
        \begin{tikzcd}
            & \Pic(\tilde{\mathcal C})\\
            \Pic(C)\arrow[ru] \arrow[rd] & \\
            & \Pic(\mathcal C)\arrow[uu].
        \end{tikzcd}
    \end{center}
    This induces a diagram
    \begin{center}
        \begin{tikzcd}
            0\arrow[r]& \Pic(C)\arrow[r] & \Pic(\tilde{\mathcal C})\arrow[r] & \mathbb Z/n \arrow[r] & 0\\
            0 \arrow[r] & \Pic(C) \arrow[r]\arrow[u,"\id"] & \Pic(\mathcal C)\arrow[r]\arrow[u, "\eta^*"] & \Hom(G^{\text{ab}},\bar{\mathbbm k}^\times) \arrow[r]\arrow[u] & 0,
        \end{tikzcd}
    \end{center}
    as those terms give the cokernels of these maps by \cite[Theorem 1.1]{Lop23} and Lemma~\ref{H2 injectivity}. The kernel of $\eta^*$ is $\mu_2$ by our exact sequence in Equation~\eqref{final exact sequence},
    and so the kernel of the right map is $\mu_2$ as well.
    We know that $\Hom(G^{\text{ab}},\bar{\mathbbm k}^\times)$ has size $2d_-$ because $\Hom(G^{\text{ab}},\bar{\mathbbm k}^\times)\cong G^{\text{ab}}$ and Proposition~\ref{Gab} says $G^{\text{ab}}$ has size $2d_-$. Then the image in $\mathbb Z/n$ has size $d_-$ by exactness.

    There is a unique subgroup of $\mathbb Z/n$ of size $d_-$, which then must be the image of $\Hom(G^{\text{ab}},\bar{\mathbbm k}^\times) \to \mathbb Z/n.$ Pulling this back to 
    \[\Pic\mathcal C\to \Pic\tilde{\mathcal C}=\Pic(C)\left\langle \sqrt[n]{\mathcal O(p)}\right\rangle,\]
    we see that its image is $\Pic(C)\left\langle \sqrt[d_-]{\mathcal O(p)}\right\rangle.$ And so we have an exact sequence
    \begin{equation}\label{doubly-twisted final exact}
        0\to \mu_2\to \Pic\mathcal C\to \Pic(C)\left\langle \sqrt[d_-]{\mathcal O(p)}\right\rangle\to 0.
    \end{equation}
    By the diagram, a splitting of this exact sequence is induced by a splitting of 
    \[0\to \mu_2\to \Hom(G^{\text{ab}},\bar{\mathbbm k}^\times)\to \mathbb Z/d_-\to 0.\]
    If $\Hom(G^{\text{ab}},\bar{\mathbbm k}^\times)\cong \mathbb Z/2\oplus \mathbb Z/d_-$, then this sequence must split, giving the description of
    $\Pic\mathcal C$ in the proposition. Otherwise, \eqref{doubly-twisted final exact} does not split. In this case, we can take a lift $\mathscr L\in \Pic\mathcal C$ of $\sqrt[d_-]{\mathcal O(p)},$ then $\mathscr L^{\otimes d_-}$ is equal to $\mathcal O(p)$ up to an element of $\mu_2$. Here we view $\mathcal O(p)$ as a line bundle on $\mathcal C$ via pullback from $C$.
    
    We cannot have $\mathscr L^{\otimes d_-}=\mathcal O(p)$ on the dot, because this implies the above exact sequence is split. But squaring will kill the nonzero part coming from $\mu_2$, giving $\mathscr L^{\otimes 2d_-}=\mathcal O(2p)$. This gives the description of $\Pic(\mathcal C)$ from the statement of the proposition.
\end{proof}

\begin{example}
From Theorem \ref{doubly-twisted Picard},
any choice of an action with $a=1$ leads to the largest possible Picard group for a nodal curve.
We consider the interesting case of $\mathbb Z/2n$: for a fixed primitive $2n^{\text{th}}$ root of unity $\zeta_{2n}$,
let $G=\mathbb Z/2n$ act as $1\cdot(x,y)=(\zeta_{2n}y,\zeta_{2n}x)$. Then this is a
doubly-twisted node whose geometric stabilizer is cyclic (yet not twisted), and
$$
\Pic\mathcal C=\Pic C\left<\sqrt[2n]{\mathcal O(2p)}\right>,
$$
because $m=1$ is odd and $d_-=n$.
\end{example}

\subsection{The Brauer group of $\mathcal C$}
In this section, all group actions on $\mathbbm k^{\times}$ are taken to be trivial.
We also assume that $\mathbbm k$ is algebraically closed.
Recall that $G$ is an extension of $\mathbb Z/2$ by $\mu_n$:
$$
1\rightarrow\mu_n\rightarrow G\rightarrow\mathbb Z/2\rightarrow1.
$$
We will use the Lyndon-Hochschild-Serre spectral sequence
$$
E_2^{p,q}=\coh^p(\mathbb Z/2,\coh^q(\mu_n,\mathbbm k^{\times}))\implies
\coh^{p+q}(G,\mathbbm k^{\times})
$$
to compute $\coh^2(G,\mathbbm k^{\times}).$

\begin{lemma}\label{kernel}
We have
$$
\coh^2(G,\mathbbm k^{\times})=\ker\left(d_2^{1,1}:\coh^1(\mathbb Z/2,\coh^1(\mu_n,\mathbbm k^{\times}))
\rightarrow\coh^3(\mathbb Z/2,\mathbbm k^{\times})\right).
$$
\end{lemma}

\begin{proof}
The $E_2^{2,0}$ and $E_2^{0,2}$ terms are
$$
E_2^{2,0}=\coh^2(\mathbb Z/2,\mathbbm k^{\times})=\mathbbm k^{\times}/(\mathbbm k^{\times})^2=0
$$
and
$$
E_2^{0,2}=\coh^0(\mathbb Z/2,\coh^2(\mu_n,\mathbbm k^{\times}))=\coh^0(\mathbb Z/2,0)=0.
$$
Therefore the end of the sequence of low-degree terms reads
$$
0\rightarrow\coh^2(G,\mathbbm k^{\times})\rightarrow\coh^1(\mathbb Z/2,\coh^1(\mu_n,\mathbbm k^{\times}))
\xrightarrow{d_2^{1,1}}\coh^3(\mathbb Z/2,\mathbbm k^{\times}),
$$
and so
$$
\coh^2(G,\mathbbm k^{\times})=\ker(d_2^{1,1}:E_2^{1,1}\rightarrow E_2^{3,0})
$$
as desired.
\end{proof}

\begin{lemma}\label{E_2^{1,1}}
The $E_2^{1,1}$ term is
$$
\coh^1(\mathbb Z/2,\coh^1(\mu_n,\mathbbm k^{\times}))=\frac{\mathbb Z/d_+}{\mathbb Z/(n/d_-)},
$$
which is either trivial if $d_-d_+=n$ or $\mathbb Z/2$ if $d_-d_+=2n$.
The generator of this group is $n_+$.
\end{lemma}

\begin{proof}
The formulas in Proposition \ref{group coho} tell us that
$$
\coh^1(\mathbb Z/2,\coh^1(\mu_n,\mathbbm k^{\times}))=
\coh^1(\mathbb Z/2,\mathbb Z/n)
$$
where the action of $\mathbb Z/2$ on $\mathbb Z/n$ is multiplication by $a$.
Further application of these formulas gives
$$
\coh^1(\mathbb Z/2,\mathbb Z/n)=\frac{\{k\in\mathbb Z/n:(a+1)k=0\}}{(a-1)\mathbb Z/n}=
\frac{\mathbb Z/d_+}{\mathbb Z/(n/d_-)}.
$$
This group is either trivial if $d_-d_+=n$ or $\mathbb Z/2$ if $d_-d_+=2n$, and its generator in either case
is the element $n_+$.
\end{proof}

\begin{note}\label{identification}
In Lemma \ref{E_2^{1,1}} we identified
$$
\coh^1(\mu_n,\mathbbm k^{\times})=\Hom(\mu_n,\mathbbm k^{\times})=\mathbb Z/n.
$$
Unravelling this identification to write in terms of $\Hom(\mu_n,\mathbbm k^{\times})$,
which will be necessary in the proof of Theorem \ref{H2 doubly twisted}, we see that the generator
is
$$
\left[\varphi:\zeta_n\mapsto\zeta_n^{n_+}\right].
$$
\end{note}

\begin{proposition}\label{H2 doubly twisted}
Suppose that $\mathcal C$ has a doubly-twisted node over an algebraically closed field.
The contribution of the node to
the Brauer group of $\mathcal C$ is
$$
\coh^2(G,\mathbbm k^{\times}_{\text{triv}})=
\begin{cases}
0 & \text{if } d_-d_+=n\text{ or $G$ is non-split}\\
\mathbb Z/2 & \text{if } d_-d_+=2n\text{ and $G$ is split.}
\end{cases}
$$
\end{proposition}

\begin{proof}
By Lemmas \ref{kernel} and \ref{E_2^{1,1}}, we may assume that $d_-d_+=2n$, else
$E_2^{1,1}$ and hence $\coh^2(G,\mathbbm k^{\times})$ is trivial.

By Note \ref{identification}, the generator of
$$
\coh^1(\mathbb Z/2,\coh^1(\mu_n,\mathbbm k^{\times}))
$$
is the element
$$
\left[\varphi:\zeta_n\mapsto\zeta_n^{n_+}\right].
$$

By Proposition~\ref{m prop}, the class of the extension $[G]$ is determined by $m$.
By \cite[Theorem 4]{HS53}, for the element
$\varphi\in\coh^1(\mathbb Z/2,\coh^1(\mu_n,\mathbbm k^{\times}))$,
we have
$$
d_2^{1,1}(\varphi)=\varphi(\zeta_n^{-m})\in\mu_2=\Hom(\mathbb Z/2,\mathbbm k^{\times})
=\coh^3(\mathbb Z/2,\mathbbm k^{\times}).
$$
This shows that if the
extension is split, then $d_2^{1,1}$ is trivial and hence
$$\coh^2(G,\mathbbm k^{\times})=
\coh^1(\mathbb Z/2,\mathbb Z/n)=\mathbb Z/2.
$$

In the non-split case, we consider
$$
d_2^{1,1}(\varphi)=\varphi(\zeta_n^{-m})=\zeta_n^{-mn_+}.
$$
Note that this is equal to 1 if and only if $d_+\mid m$, which is precisely the criterion
for whether or not the extension is trivial from Proposition \ref{m prop}.
Hence $d_2^{1,1}(\varphi)\neq 1$ when $d_-d_+=2n$ and the extension is nontrivial.
We conclude that $d_2^{1,1}$ is injective and $\coh^2(G,\mathbbm k^{\times})=0$.
\end{proof}

\begin{example}
Theorem \ref{H2 doubly twisted} now provides the example promised in \cite[Example 3.9]{Bis25}.
Pick any split extension $G=\mu_n\rtimes\mathbb Z/2$ with $a$ chosen so that
$d_-d_+=2n$. For $\mathbbm k$ algebraically closed, set
$$
\mathcal C=\left[\frac{\Spec\mathbbm k[x,y]/(xy)}G\right]
$$
where $G$ acts as $(\zeta_n,\zeta_n^a)$ from $\mu_n$ and swaps branches from $\mathbb Z/2$.
Then
$$
\coh^2(\mathcal C,\mathbb G_m)=\coh^2(G,\mathbbm k^{\times})=\mathbb Z/2.
$$
For an example of such an $a$, we can pick $a=-1$. Hence $G=D_n$, and so
$\mathcal C$ is a balanced doubly-twisted curve (or is of dihedral type).
Any such $a$ and $G$ provide, from the perspective of complexity of the singularity,
the simplest possible example of
a tame stacky curve over an algebraically closed field with nontrivial Brauer group.
\end{example}

\section{Twisted nodes}\label{twisted node section}
\begin{convention}\label{twisted alg closed}
In this section we fix a stacky nodal curve $\mathcal C$ with a
twisted node, $p$. Denote the geometric stabilizer at $p$, which is $\mu_n$, by $G$.
Let $\pi:\mathcal C\rightarrow C$ be the coarsening at $p$, let
$\eta:\tilde{\mathcal C}\rightarrow\mathcal C$ be the partial normalization
of $\mathcal C$ at $p$, and let $\tilde C$ be the
coarsening of $\tilde{\mathcal C}$ at the preimages of the node
(which is also the partial normalization of $C$).
Lastly, recall that we have set
$d_+=\gcd(a+1,n)$, $d_-=\gcd(a-1,n)$, and $n_+=n/d_+$.
\end{convention}

\begin{lemma}\label{twisted Leray}
The below sequence is exact:
$$
0\rightarrow\Pic C\rightarrow\Pic\mathcal C
\rightarrow\mathbb \Pic B\mu_n\rightarrow0.
$$
\end{lemma}

\begin{proof}
The beginning of the sequence of low-degree terms
of the Leray spectral sequence for $\pi:\mathcal C\rightarrow C$ reads
\[
0\rightarrow\Pic C\rightarrow\Pic\mathcal C\rightarrow
\coh^0(C,\mathbf R^1\pi_*\mathbb G_m)\rightarrow \coh^2(C,\mathbb G_m)\rightarrow\coh^2(\mathcal C,\mathbb G_m).
\]
As in Section \ref{doubly-twisted section}, $\mathbf R^1\pi_*\mathbb G_m$
vanishes except at the node, where its stalk is the Picard group of the node, $\Pic B\mu_n=\mathbb Z/n$. Hence
\[
0\rightarrow\Pic C\rightarrow\Pic\mathcal C\rightarrow
\mathbb \Pic B\mu_n\rightarrow \coh^2(C,\mathbb G_m)\rightarrow\coh^2(\mathcal C,\mathbb G_m)
\]
is exact.

Therefore it suffices to show the injectivity of
$$
\coh^2(C,\mathbb G_m)\rightarrow\coh^2(\mathcal C,\mathbb G_m).
$$
Following Lemma \ref{H2 injectivity}, away from the node we have
$C\setminus\{\pi(p)\}=U=\mathcal C\setminus\{p\}$.
Consider the diagram
$$
\begin{tikzcd}
\coh^2(C,\mathbb G_m)\arrow[rd]\arrow[rr] & & \coh^2(\mathcal C,\mathbb G_m)\arrow[ld]\\
 & \coh^2(U,\mathbb G_m). &
\end{tikzcd}
$$
To show that $\coh^2(C,\mathbb G_m)\rightarrow\coh^2(\mathcal C,\mathbb G_m)$
is injective, it suffices to show that the restriction of
$\coh^2(C,\mathbb G_m)$ to $U$ is injective. However, since the node is split,
the kernel of this restriction is precisely
$$
\ker\left[\Br\mathbbm k\rightarrow\bigoplus_{i=1}^2\Br\mathbbm k\right],
$$
where the map into each direct summand is the identity. Therefore
the restriction and hence pullback is injective.
\end{proof}

Following Section \ref{doubly-twisted section}, let $Q$ be the quotient
$$
0\rightarrow\mathcal O_{\mathcal C}^{\times}\rightarrow\eta_*\mathcal O_{\tilde{\mathcal C}}^{\times}
\rightarrow Q\rightarrow0.
$$
We get the exact sequence
{\[
0\rightarrow\mathcal O_{\mathcal C}(\mathcal C)^\times \to \mathcal O_{\tilde {\mathcal C}}(\tilde {\mathcal C})^\times\to \coh^0(\mathcal C,Q)\rightarrow\Pic\mathcal C\rightarrow
\Pic\tilde{\mathcal C}\rightarrow\coh^1(\mathcal C,Q).
\]} 

\begin{proposition}\label{Q structure twist}
Let $\mathbbm k^{\times}_{\text{triv}}$ denote the sheaf on $BG$ associated to the $G$-module whose
underlying group is $\mathbbm k^\times$, where the action is trivial.
Then $Q\cong p_* \mathbbm k^\times_{\text{triv}}$.
\end{proposition}

\begin{proof}
This is identical to Proposition \ref{Q structure}, except now the action of $G$ on
$\mathbbm k^{\times}$ is trivial.
\end{proof}

\begin{corollary}
The cohomology of $Q$ is $\coh^0(\mathcal C,Q)=\mathbbm k^{\times}$ and
$\coh^1(\mathcal C,Q)=\mathbb Z/n$.
\end{corollary}

\begin{proof}
This is also identical to its Section \ref{doubly-twisted section} counterpart, now with trivial action.
Hence
$$
\coh^0(\mathcal C,Q)=\coh^0(G,\mathbbm k^{\times})=\mathbbm k^{\times}
$$
and
\[
\coh^1(\mathcal C,Q)=\coh^1(\mu_n,\mathbbm k^{\times})=\mathbb Z/n. \qedhere
\]
\end{proof}

The long exact sequence in cohomology then becomes the following.
\begin{corollary}\label{twisted Pic sequence}
If the node is nonseparating, our exact sequence is
\[
0\rightarrow\mathbbm k^{\times}\rightarrow\Pic\mathcal C\rightarrow\Pic\tilde{\mathcal C}\rightarrow
\mathbb Z/n.
\]
Otherwise, if the node is separating, the exact sequence is
\[
0\to\Pic\mathcal C\rightarrow\Pic\tilde{\mathcal C}\rightarrow
\mathbb Z/n.
\]
\end{corollary}

\begin{note}
When $n=1$, so that $\mathcal C=C$ and $\tilde{\mathcal C}=\tilde C$, this gives the well-known exact sequences for nodes with no stacky structure
\[
    0\to \mathbbm k^\times \to \Pic C\to \Pic \tilde C\to 0
\]
for nonseparating nodes and
\[
0\to \Pic C\to \Pic\tilde{C}\to 0
\]
for separating nodes.
\end{note}

\begin{proof}[Proof of Corollary \ref{twisted Pic sequence}]
    If the node is nonseparating, then $\mathcal C$ and $\tilde{\mathcal C}$ are geometrically connected, so $\mathcal O_{\mathcal C}(\mathcal C)^\times \to \mathcal O_{\tilde {\mathcal C}}(\tilde {\mathcal C})^\times$ is an isomorphism. This gives the desired exact sequence.

    If the node is separating, we have $\mathcal O_{\mathcal C}(\mathcal C)^\times=\mathbbm k^\times$, but $\mathcal O_{\tilde {\mathcal C}}(\tilde {\mathcal C})^\times=(\mathbbm k^{\times})^2$, and the morphism between them is the diagonal. Then, the beginning of the exact sequence looks like
    \[0\to \mathbbm k^\times\to (\mathbbm k^\times)^2\to \mathbbm k^{\times}.\]
    Here, the second map sends $(a,b)\to ab^{-1}$. Then this sequence is short exact, which gives our desired exact sequence.
\end{proof}

\begin{proposition}\label{twisted image}
    The image of $\Pic\mathcal C\to \Pic\tilde{\mathcal C}$ is given by
    \[\Pic\tilde C\left<\sqrt[n]{\mathcal O(p_1)}\otimes \sqrt[n]{\mathcal O(p_2)}^{\otimes a}\right>,\]
    where we view $\Pic\tilde C\left<\sqrt[n]{\mathcal O(p_1)}\otimes \sqrt[n]{\mathcal O(p_2)}^{\otimes a}\right>$ inside of 
$$
\Pic\tilde{\mathcal C}=\Pic\tilde C\left\langle \sqrt[n]{\mathcal O(p_1)}, \sqrt[n]{\mathcal O(p_2)}\right\rangle.
$$
\end{proposition}
\begin{proof}
We have a diagram
$$
    \begin{tikzcd}
        0\arrow[r] & \Pic\tilde C\arrow[r] & \Pic\tilde{\mathcal C}\arrow[r] & (\Pic B\mu_n)^2\arrow[r] & 0\\
        0\arrow[r] & \Pic C\arrow[r]\arrow[u] & \Pic\mathcal C\arrow[r]\arrow[u] & \Pic B\mu_n\arrow[r]\arrow[u] & 0
    \end{tikzcd}
$$

The top row is exact by \cite[Theorem 1.1]{Lop23} and the bottom row is exact by Lemma~\ref{twisted Leray}. By the \'{e}tale local description of $\mathcal C$ at $p$, the pullback of $1\in \mathbb Z/n=\Pic B\mu_n$ to
$(\Pic B\mu_2)^2\cong (\mathbb Z/n)^2$ is $(1,a)$.  

By Corollary~\ref{twisted Pic sequence}, we have that the left map in the diagram is surjective.
Choosing an arbitrary lift $L\in \Pic \mathcal C$ of $1\in \Pic B\mu_n$, a diagram chase shows that
we may alter $L$ by an element of $\Pic C$ to assume that $L$ pulls back to
$\sqrt[n]{\mathcal O(p_1)}\otimes \sqrt[n]{\mathcal O(p_2)}^{\otimes a}$.
Thus, the image of $\Pic\mathcal{C}$ in $\Pic\tilde{\mathcal{C}}$ contains
$\Pic\tilde C\left<\sqrt[n]{\mathcal O(p_1)}\otimes \sqrt[n]{\mathcal O(p_2)}^{\otimes a}\right>$.
As this group has index $n$ in $\Pic\tilde{\mathcal C}$, it must be the entirety of the image by
Corollary~\ref{twisted Pic sequence}.
\end{proof}

Thus, we have proved the following theorem.
\begin{theorem}\label{twisted Picard v1}
If $p$ is nonseparating, we have an exact sequence
$$
0\rightarrow\mathbbm k^{\times}\rightarrow\Pic\mathcal C\rightarrow \Pic\tilde C\left<\sqrt[n]{\mathcal O(p_1+ap_2)}\right>\rightarrow 0.
$$
If $p$ is separating, we have
\[\Pic\mathcal C=\Pic\tilde C\left<\sqrt[n]{\mathcal O(p_1+ap_2)}\right>.\]
\end{theorem}

Clearly the behavior of these sequences depends on many factors,
including $a$, $p_1$ and $p_2$, the structure of $C$ and $\tilde C$,
and even the base field $\mathbbm k$. We give here the most important special case.

\begin{corollary}\label{balanced twisted Picard}
Assume that $\mathbbm k$ is algebraically closed and that $\mathcal C$ has a
balanced stable node, that is, $a\equiv-1\mod n$. If $p$ is nonseparating, then
$$
\Pic\mathcal C\cong\mathbbm k^{\times}\oplus\Pic\tilde{C}\oplus\mathbb Z/n\cong\Pic C\oplus\mathbb Z/n,
$$
and if $p$ is separating
\[\Pic\mathcal C\cong \Pic\tilde C\oplus \mathbb Z/n= \Pic C\oplus \mathbb Z/n.\]
\end{corollary}
\begin{proof}
Because $\mathbbm k$ is algebraically closed, $\Pic^0(\tilde C)$ is divisible, so we can find some
$\mathscr L$ with $\mathscr L^{\otimes n}\cong \mathcal O(p_1-p_2)$. Because
$\Pic C\to \Pic \tilde C$ is surjective, there is some element of $\Pic C$ pulling back to
$\mathscr L$. Let $\mathscr M$ be this element pulled back to $\Pic \mathcal C$.

By Proposition~\ref{twisted image}, there is a line bundle $\mathcal N\in \Pic\mathcal C$
pulling back to
$$
\sqrt[n]{\mathcal O(p_1)}\otimes \sqrt[n]{\mathcal O(p_2)}^{\vee}\in \Pic\tilde{\mathcal C}.
$$
Then $(\mathcal M\otimes \mathcal N)^{\otimes n}$ pulls back to $\mathcal O_{\tilde{\mathcal C}}$,
and so by Corollary~\ref{twisted Pic sequence}, $(\mathcal M\otimes \mathcal N)^{\otimes n}$
is in the image of $\mathbbm k^\times \to \Pic\mathcal C$ if the node is nonseparating, and is $\mathcal O$ otherwise. Then, using that $\mathbbm k^\times$ is divisible for the nonseparating case, we can find $\mathcal P$ pulled back from $C$ so that $\mathcal P^{\otimes n}=(\mathcal M\otimes \mathcal N)^{\otimes n}$, so
\[
(\mathcal M\otimes \mathcal P^\vee \otimes \mathcal N)^{\otimes n}=\mathcal O_{\mathcal C}.
\]
And, because $\mathcal M\otimes \mathcal P^\vee$ is pulled back from
$\Pic C\to \Pic\mathcal C$, $\mathcal M\otimes \mathcal P^\vee \otimes \mathcal N$ lifts a generator of
$\Pic B\mu_n$, and hence the exact sequence in Lemma~\ref{twisted Leray} is split.
Thus, we have
\[
\Pic\mathcal C\cong \Pic C\oplus\mathbb Z/n.
\]
In the separating case, $\Pic C=\Pic\tilde C$, and we are done.
In the nonseparating case, because $\mathbbm k$ is algebraically closed, $\mathbbm k^\times$ is divisible,
so the exact sequence in Corollary~\ref{twisted Pic sequence} splits, giving
\[
\Pic\mathcal C\cong \Pic C\oplus\mathbb Z/n\cong \mathbbm k^\times \oplus \Pic\tilde C\oplus \mathbb Z/n.
\qedhere\]
\end{proof}

Finally, we note that a twisted node does not contribute to the Brauer group.
\begin{proposition}\label{H2 twisted}
Let $\mathcal C$ be a stacky curve over an algebraically closed field
with a twisted node. The contribution of the node to $\coh^2(\mathcal C,\mathbb G_m)$
is
$$\coh^2(G,\mathbbm k^{\times}_{\text{triv}})=0.
$$
\end{proposition}

\begin{proof}
This follows from Proposition \ref{group coho}, since cyclic groups have vanishing second cohomology
with coefficients in $\mathbbm k^{\times}_{\text{triv}}$.
\end{proof}

\newpage
\bibliographystyle{alpha}

\begin{thebibliography}{ATWo20}

\bibitem[Ach24]{Ach24}
Niven Achenjang.
\newblock On Brauer groups of tame stacks, 2024.

\bibitem[ACV03]{ACV03}
Dan Abramovich, Alessio Corti, and Angelo Vistoli.
\newblock Twisted bundles and admissible covers.
\newblock volume~31, pages 3547--3618. 2003.
\newblock Special issue in honor of Steven L. Kleiman.

\bibitem[AOV11]{AOV11}
Dan Abramovich, Martin Olsson, and Angelo Vistoli.
\newblock Twisted stable maps to tame {A}rtin stacks.
\newblock {\em J. Algebraic Geom.}, 20(3):399--477, 2011.

\bibitem[ATWo20]{ATW20}
Dan Abramovich, Michael Temkin, and Jaros\l~aw W\l~odarczyk.
\newblock Toroidal orbifolds, destackification, and {K}ummer blowings up.
\newblock {\em Algebra Number Theory}, 14(8):2001--2035, 2020.
\newblock With an appendix by David Rydh.

\bibitem[AV02]{AV02}
Dan Abramovich and Angelo Vistoli.
\newblock Compactifying the space of stable maps.
\newblock {\em J. Amer. Math. Soc.}, 15(1):27--75, 2002.

\bibitem[Bis25]{Bis25}
Martin Bishop.
\newblock Brauer groups of tame stacky curves and their $\mu_r$-gerbes, 2025.

\bibitem[DF04]{DF04}
David~S. Dummit and Richard~M. Foote.
\newblock {\em Abstract algebra}.
\newblock John Wiley \& Sons, Inc., Hoboken, NJ, third edition, 2004.

\bibitem[DHW12]{DHW12}
Karel Dekimpe, Manfred Hartl, and Sarah Wauters.
\newblock A seven-term exact sequence for the cohomology of a group extension.
\newblock {\em J. Algebra}, 369:70--95, 2012.

\bibitem[Eke95]{Eke95}
Torsten Ekedahl.
\newblock Boundary behaviour of {H}urwitz schemes.
\newblock In {\em The moduli space of curves ({T}exel {I}sland, 1994)}, volume
  129 of {\em Progr. Math.}, pages 173--198. Birkh\"auser Boston, Boston, MA,
  1995.

\bibitem[Fuc70]{Fuc70}
L\'aszl\'o Fuchs.
\newblock {\em Infinite abelian groups. {V}ol. {I}}, volume Vol. 36 of {\em
  Pure and Applied Mathematics}.
\newblock Academic Press, New York-London, 1970.

\bibitem[HS53]{HS53}
G.~Hochschild and J.-P. Serre.
\newblock Cohomology of group extensions.
\newblock {\em Trans. Amer. Math. Soc.}, 74:110--134, 1953.

\bibitem[Lop23]{Lop23}
Rose Lopez.
\newblock Picard groups of stacky curves.
\newblock 2023.

\bibitem[Mei18]{Mei18}
Lennart Meier.
\newblock Computing Brauer groups using coarse moduli -- draft version.
\newblock 2018.

\bibitem[Rom05]{Rom05}
Matthieu Romagny.
\newblock Group actions on stacks and applications.
\newblock {\em Michigan Math. J.}, 53(1):209--236, 2005.

\bibitem[{Sta}25]{Sta}
The {Stacks Project Authors}.
\newblock \textit{Stacks Project}.
\newblock \url{https://stacks.math.columbia.edu}, 2025.

\bibitem[Wei94]{Wei94}
Charles~A. Weibel.
\newblock {\em An introduction to homological algebra}, volume~38 of {\em
  Cambridge Studies in Advanced Mathematics}.
\newblock Cambridge University Press, Cambridge, 1994.

\end{thebibliography}

\end{document}